\documentclass[12pt,twoside, final]{amsart}
\copyrightinfo{0}{Iranian Mathematical Society}
\usepackage{amsmath,amsthm,amscd,amsfonts,amssymb,enumerate}
\usepackage{graphicx}		
\usepackage{color}
\usepackage[colorlinks]{hyperref}

\newtheorem{theorem}{Theorem}[section]

\newtheorem{lemma}[theorem]{Lemma}
\newtheorem{corollary}[theorem]{Corollary}
\theoremstyle{definition}

\newtheorem{example}[theorem]{Example}
\theoremstyle{remark}

\numberwithin{equation}{section}

\def\R{\hbox{\bf\rlap{I}{\hbox to 2 pt{}}R}}

\def\stb#1#2{\mathop{#1\vphantom{\sum}}\limits_{\scriptstyle #2}}
\def\btb#1#2{\mathop{#1\vphantom{\int}}\limits_{\scriptstyle #2}}

 \commby{Hamid Pezeshk}
\begin{document}


\title[Upper and Lower Bounds for Numerical Radii of Block Shifts]{Upper and Lower Bounds for Numerical Radii of Block Shifts}

\author[H.-L. Gau]{Hwa-Long Gau}
\address[Hwa-Long Gau]{Department of Mathematics, National Central University, Chungli 32001, Taiwan}
\email{hlgau@math.ncu.edu.tw}

\author[P.Y. Wu]{Pei Yuan Wu$^*$}
\address[Pei Yuan Wu]{Department of Applied Mathematics, National Chiao Tung University, Hsinchu 30010, Taiwan}
\email{pywu@math.nctu.edu.tw}

\dedicatory{Dedicated to Professor Heydar Radjavi on his 80th birthday}

  \thanks{$^*$Corresponding author}
%
\maketitle
%

\begin{abstract}
For any $n$-by-$n$ matrix $A$ of the form
\[\left[\begin{array}{cccc} 0 & A_1 & & \\ & 0 & \ddots & \\ & & \ddots & A_{k-1} \\ & & & 0\end{array}\right],\]
we consider two $k$-by-$k$ matrices
\[A'=\left[\begin{array}{cccc} 0 & \|A_1\| & & \\ & 0 & \ddots & \\ & & \ddots & \|A_{k-1}\| \\ & & & 0\end{array}\right] \  \mbox{ and } \
A''=\left[\begin{array}{cccc} 0 & m(A_1) & & \\ & 0 & \ddots & \\ & & \ddots & m(A_{k-1}) \\ & & & 0\end{array}\right],\]
where $\|\cdot\|$ and $m(\cdot)$ denote the operator norm and minimum modulus of a matrix, respectively. It is shown that the numerical radii $w(\cdot)$ of $A$, $A'$ and $A''$ are related by the inequalities $w(A'')\le w(A)\le w(A')$. We also determine exactly when either of the inequalities becomes an equality.
\end{abstract}

\section{\bf Introduction}

An $n$-by-$n$ complex matrix $A$ is call a \emph{block shift} if it is of the form
\[\left[\begin{array}{cccc} 0 & A_1 & & \\ & 0 & \ddots & \\ & & \ddots & A_{k-1} \\ & & & 0\end{array}\right],\]
where the $A_j$'s are in general rectangular matrices. In this paper, we obtain sharp upper and lower bounds for the numerical radius $w(A)$ of such an $A$. Recall that the \emph{numerical radius} $w(B)$ of an $n$-by-$n$ matrix $B$ is the quantity
\[\max\{|\langle Bx, x\rangle| : x\in\mathbb{C}^n, \|x\|=1\},\]
where $\langle\cdot, \cdot\rangle$ and $\|\cdot\|$ denote the standard inner product and norm of vectors in $\mathbb{C}^n$, respectively. Note that $w(B)$ is the radius of the smallest circular disc centered at the origin which contains the \emph{numerical range}
\[W(B)=\{\langle Bx, x\rangle : x\in\mathbb{C}^n, \|x\|=1\}\]
of $B$. For properties of the numerical range and numerical radius, the reader is referred to \cite[Chapter 22]{3} or \cite[Chapter 1]{4}.

\vspace{3mm}

Note that if $A$ is a block shift of the above form, then it is unitarily similar to $e^{i\theta}A$ for all real $\theta$. Hence its numerical range is a closed circular disc centered at the origin with radius equal to its numerical radius. To estimate the latter, we consider two $k$-by-$k$ scalar matrices
\[A'=\left[\begin{array}{cccc} 0 & \|A_1\| & & \\ & 0 & \ddots & \\ & & \ddots & \|A_{k-1}\| \\ & & & 0\end{array}\right] \  \mbox{ and } \
A''=\left[\begin{array}{cccc} 0 & m(A_1) & & \\ & 0 & \ddots & \\ & & \ddots & m(A_{k-1}) \\ & & & 0\end{array}\right],\]
where $\|A_j\|$ and $m(A_j)$, $1\le j\le k-1$, are the operator norm and minimum modulus of $A_j$, respectively. Recall that the minimum modulus $m(B)$ of an $m$-by-$n$ matrix $B$ is, by definition, $\min\{\|Bx\| : x\in\mathbb{C}^n, \|x\|=1\}$. In Sections \ref{s2} and \ref{s3} below, we show that $w(A'')\le w(A)\le w(A')$ always hold, and that, under the extra condition that the $A_j$'s are all nonzero (resp., under $A_1\ldots A_{k-1}\neq 0$), $w(A)=w(A')$ (resp., $w(A)=w(A'')$) implies that $A'$ (resp., $A''$) is a direct summand of $A$ (cf. Theorems \ref{t21} and \ref{t31}). Examples are given showing that the nonzero conditions on the $A_j$'s are essential.

\vspace{3mm}

\section{\bf Upper bound}\label{s2}

The main result of this section is the following theorem.

\begin{theorem}\label{t21}
Let
\begin{equation}
A=\left[\begin{array}{cccc} 0 & A_1 & & \\ & 0 & \ddots & \\ & & \ddots & A_{k-1} \\ & & & 0\end{array}\right] \ \ \mbox{ on } \ \mathbb{C}^n=\mathbb{C}^{n_1}\oplus\cdots\oplus \mathbb{C}^{n_k} \label{eq1}
\end{equation}
be an $n$-by-$n$ block shift, where $A_j$ is an $n_j$-by-$n_{j+1}$ matrix for $1\le j\le k-1$, and let
\[A'=\left[\begin{array}{cccc} 0 & \|A_1\| & & \\ & 0 & \ddots & \\ & & \ddots & \|A_{k-1}\| \\ & & & 0\end{array}\right] \ \ \mbox{ on } \ \mathbb{C}^k.\]
Then {\rm (a)} $w(A)\le w(A')$, and {\rm (b)} under the assumption that $A_j\neq 0$ for all $j$, $w(A)=w(A')$ if and only if $A$ is unitarily similar to $A'\oplus B$, where $B$ is a block shift with $w(B)\le w(A')$.
\end{theorem}

\begin{proof}
(a) Let $x=[x_1 \ \ldots \ x_k]^T$ be a unit vector in $\mathbb{C}^n$ such that $|\langle Ax, x\rangle|=w(A)$. Hence
\begin{align}
& \, w(A)=\left|\Big\langle\left[\begin{array}{cccc} 0 & A_1 & & \\ & 0 & \ddots & \\ & & \ddots & A_{k-1} \\ & & & 0\end{array}\right]\left[\begin{array}{c} x_1\\ \vdots \\ x_k\end{array}\right], \left[\begin{array}{c} x_1\\ \vdots \\ x_k\end{array}\right]\Big\rangle\right|\nonumber\\
= & \, \Big|\sum_{j=1}^{k-1}\langle A_jx_{j+1}, x_j\rangle\Big|\nonumber\\
\le & \, \sum_{j=1}^{k-1}|\langle A_jx_{j+1}, x_j\rangle|\nonumber\\
\le & \, \sum_{j=1}^{k-1}\|A_j\|\|x_{j+1}\|\|x_j\| \label{eq2}\\
= & \, \Big\langle\left[\begin{array}{cccc} 0 & \|A_1\| & & \\ & 0 & \ddots & \\ & & \ddots & \|A_{k-1}\| \\ & & & 0\end{array}\right]\left[\begin{array}{c} \|x_1\|\\ \vdots \\ \|x_k\|\end{array}\right], \left[\begin{array}{c} \|x_1\|\\ \vdots \\ \|x_k\|\end{array}\right]\Big\rangle\nonumber\\
\le & \, w(A'),\label{eq3}
\end{align}
where the last inequality follows from the fact that $[\|x_1\| \ \ldots \ \|x_k\|]^T$ is a unit vector in $\mathbb{C}^k$.

\vspace{3mm}

(b) Assume that $A_j\neq 0$ for all $j$, and that $w(A)=w(A')$. Then we have equalities throughout the chain of inequalities in (a). Since $A'$ is an (entrywise) nonnegative matrix with irreducible real part, the equality in \eqref{eq3} yields, by \cite[Proposition 3.3]{5}, that $x_j\neq 0$ for all $j$. Let $\widehat{x}_j=[0 \ \ldots \ 0 \ \stb{x_j}{j{\rm th}} \ 0 \ \ldots \ 0]^T$ for $1\le j\le k$, and let $K$ be the subspace of $\mathbb{C}^n$ spanned by the $\widehat{x}_j$'s. The equality in \eqref{eq2} implies that
\begin{equation}
|\langle A_jx_{j+1}, x_j\rangle|=\|A_jx_{j+1}\|\|x_j\|=\|A_j\|\|x_{j+1}\|\|x_j\|.\label{eq4}
\end{equation}
Hence $A_jx_{j+1}=a_jx_j$ for some scalar $a_j$. Therefore, $A\widehat{x}_1=0$ and
\[A\widehat{x}_j=[0 \ \ldots \ 0 \ \stb{A_{j-1}x_j}{(j-1){\rm st}} \ 0 \ \ldots \ 0]^T=[0 \ \ldots \ 0 \ \stb{a_{j-1}x_{j-1}}{(j-1){\rm st}} \ 0 \ \ldots \ 0]^T=a_{j-1}\widehat{x}_{j-1}\]
is in $K$ for all $j$, $2\le j\le k$. This shows that $AK\subseteq K$.

\vspace{3mm}

We next prove that $A^*K\subseteq K$. Indeed, we have $A^*\widehat{x}_j=[0 \ \ldots \ 0 \ \stb{A_{j}^*x_j}{(j+1){\rm st}} \ 0 \ \ldots \ 0]^T$ for $1\le j\le k-1$. Since
\[|a_j|\|x_j\|^2=\|a_jx_j\|\|x_j\|=\|A_jx_{j+1}\|\|x_j\|=\|A_j\|\|x_{j+1}\|\|x_j\|\]
by \eqref{eq4}, the nonzeroness of the $A_j$'s and $x_j$'s yields the same for the $a_j$'s. Letting $B_j=A_j/\|A_j\|$ and $y_j=(\|A_j\|/a_j)x_{j+1}$, we have $B_jy_j=(1/a_j)A_jx_{j+1}=x_j$ with $\|B_j\|=1$ and
\[\|y_j\|=\frac{\|A_j\|}{|a_j|}\|x_{j+1}\|=\frac{\|A_jx_{j+1}\|}{|a_j|}=\|x_j\|\]
by \eqref{eq4}. It follows from an extended lemma of Riesz and Sz.-Nagy that $B_j^*x_j=y_j$ (cf. \cite[p. 215]{7}). Therefore, we have $A_j^*x_j=(\|A_j\|^2/a_j)x_{j+1}$, which shows that $A_j^*\widehat{x}_j=(\|A_j\|^2/a_j)\widehat{x}_{j+1}$ is in $K$ for $1\le j\le k-1$. Moreover, we also have $A^*\widehat{x}_k=0$. Thus $A^*K\subseteq K$ as asserted.

\vspace{3mm}

Since $\{\widehat{x}_j/\|x_j\|\}_{j=1}^k$ is an orthonormal basis of $K$, $A(\widehat{x}_1/\|x_1\|)=0$, and
\begin{align*}
& \, A(\frac{\widehat{x}_j}{\|x_j\|})=\frac{a_{j-1}\|x_{j-1}\|}{\|x_j\|}\frac{\widehat{x}_{j-1}}{\|x_{j-1}\|}
=\frac{a_{j-1}}{|a_{j-1}|}\frac{\|a_{j-1}x_{j-1}\|}{\|x_j\|}\frac{\widehat{x}_{j-1}}{\|x_{j-1}\|}\\
=& \, \frac{a_{j-1}}{|a_{j-1}|}\frac{\|A_{j-1}x_{j}\|}{\|x_j\|}\frac{\widehat{x}_{j-1}}{\|x_{j-1}\|}
=\frac{a_{j-1}}{|a_{j-1}|}\|A_{j-1}\|\frac{\widehat{x}_{j-1}}{\|x_{j-1}\|}
\end{align*}
for $2\le j\le k$ by \eqref{eq4}, we derive that the restriction $A|K$ is unitarily similar to $A'$. Thus $A$ is unitarily similar to $A'\oplus(A|K^{\perp})$. We now show that $A|K^{\perp}$ is also unitarily similar to a block shift. Indeed, let $\widehat{H}_j=0\oplus\cdots\oplus 0\oplus \stb{\mathbb{C}^{n_j}}{j{\rm th}}\oplus 0\oplus\cdots\oplus 0$, $K_j=\mathbb{C}^{n_j}\ominus\bigvee\{x_j\}$, and $\widehat{K}_j=0\oplus\cdots\oplus 0\oplus \stb{K_j}{j{\rm th}}\oplus 0\oplus\cdots\oplus 0$ for $1\le j\le k$. Then $K^{\perp}=K_1\oplus\cdots\oplus K_k$. Since $A\widehat{H}_{j+1}\subseteq\widehat{H}_j$ and $A^*\widehat{x}_j\in\bigvee\{\widehat{x}_{j+1}\}$ from before, we have $A\widehat{K}_{j+1}\subseteq \widehat{K}_j$ for $1\le j\le k-1$. Moreover, $A\widehat{H}_k=\{0\}$ implies that $A\widehat{K}_k=\{0\}$. We conclude that $B\equiv A|K^{\perp}$ is unitarily similar to a block shift with $w(B)\le w(A)=w(A')$. This proves one direction of (b). The converse is trivial.
\end{proof}

\vspace{1mm}

\begin{corollary}\label{c22}
Let $A$ be an $n$-by-$n$ block shift as in \eqref{eq1}, and let $M=\max_j\|A_j\|$. Then

{\rm (a)} $w(A)\le M\cdot\cos(\pi/(k+1))$, and

{\rm (b)} $w(A)=M\cdot\cos(\pi/(k+1))$ if and only if $A$ is unitarily similar to $(M\cdot J_k)\oplus B$, where $B$ is a block shift with $w(B)\le M\cdot\cos(\pi/(k+1))$.
\end{corollary}

\vspace{3mm}

Here $J_k$ denotes the $k$-by-$k$ \emph{Jordan block}
\[\left[\begin{array}{cccc} 0 & 1 & & \\ & 0 & \ddots & \\ & & \ddots & 1 \\ & & & 0\end{array}\right],\]
whose numerical range is known to be $\{z\in \mathbb{C}: |z|\le\cos(\pi/(k+1))\}$ (cf. \cite{6}).

\vspace{3mm}

\begin{proof}[Proof of Corollary \ref{c22}] (a) is an easy consequence of Theorem \ref{t21} (a) and \cite[Lemma 5 (1)]{8} while (b) follows from Theorem \ref{t21} (b) and \cite[Lemma 5 (2)]{8}.
\end{proof}

\vspace{3mm}

We remark that the assertion in Theorem \ref{t21} (b) still holds for $n\le 5$ even without the nonzero assumption on the $A_j$'s. This can be proven via a case-by-case verification by invoking, in most cases, the known result on the numerical ranges of square-zero matrices (cf. \cite[Theorem 2.1]{9}), which we omit. This is no longer the case for $n\ge 6$. Here we give a counterexample for $n=6$.

\vspace{3mm}

\begin{example}\label{e23}
Let
\[A=\left[\begin{array}{cccccc}
0 & \sqrt{2} & & & & \\
 & 0 & 0 & & & \\
 & & 0 & 1 & 0 & \\
 & & & 0 & 0 & 0 \\
 & & & 0 & 0 & 1 \\
 & & & & & 0
\end{array}\right]\]
with $A_1=[\sqrt{2}]$, $A_2=[0]$, $A_3=[1 \ 0]$ and $A_4=\left[\begin{array}{c} 0\\ 1\end{array}\right]$. Then
\[A'=\left[\begin{array}{ccccc}
0 & \sqrt{2} & & &  \\
 & 0 & 0 & &  \\
 & & 0 & 1 &  \\
 & & & 0 & 1  \\
 & & &  & 0
\end{array}\right],\]
and $A$ and $A'$ are unitarily similar to
\[\left[\begin{array}{cc} 0 & \sqrt{2} \\ 0 & 0 \end{array}\right]\oplus
\left[\begin{array}{cc} 0 & 1 \\ 0 & 0\end{array}\right]\oplus
\left[\begin{array}{cc} 0 & 1 \\ 0 & 0\end{array}\right] \ \ \mbox{ and } \ \
\left[\begin{array}{cc} 0 & \sqrt{2} \\ 0 & 0\end{array}\right]\oplus
\left[\begin{array}{ccc} 0 & 1 & \\ & 0 & 1 \\ & & 0\end{array}\right],\]
respectively. Hence $w(A)=w(A')=\sqrt{2}/2$, but $A'$ is not a direct summand of $A$. To see the latter, note that $\ker A\cap\ker A^*=\{0\}$. Hence $A$ cannot have the 1-by-1 zero matrix $[0]$ as a direct summand, and thus $A$ cannot be unitarily similar to $A'\oplus [0]$, or $A'$ is not a direct summand of $A$.
\end{example}

\vspace{3mm}

\section{\bf Lower bound}\label{s3}

Here is the main result of this section.

\begin{theorem}\label{t31}
Let $A$ be an $n$-by-$n$ block shift as in \eqref{eq1}, and let
\[A''=\left[\begin{array}{cccc} 0 & m(A_1) & & \\ & 0 & \ddots & \\ & & \ddots & m(A_{k-1}) \\ & & & 0\end{array}\right] \ \ \mbox{ on } \ \mathbb{C}^k.\]
Then {\rm (a)} $w(A)\ge w(A'')$, and {\rm (b)} under the assumption that $A_1\ldots A_{k-1}\neq 0$, $w(A)=w(A'')$ if and only if $A$ is unitarily similar to $A''\oplus C$, where $C$ is a block shift with $w(C)\le w(A'')$.
\end{theorem}

\vspace{3mm}

Our first lemma gives some basic properties of the minimum modulus of a rectangular matrix. For a square matrix (or, for that matter, an operator on a possibly infinite-dimensional Hilbert space), these appeared in \cite[Theorem 1]{2}.

\vspace{3mm}

\begin{lemma}\label{l32}
Let $A$ be an $m$-by-$n$ matrix. Then

{\rm (a)} $m(A)>0$ if and only if $A$ is left invertible, and

{\rm (b)} $m(A)$ equals the minimum singular value of $A$. In particular, if $m<n$, then $m(A)=0$.
\end{lemma}

\begin{proof}
(a) Note that $m(A)>0$ means that there is a $c>0$ such that $\|Ax\|\ge c\|x\|$ for all $x$ in $\mathbb{C}^n$, which is equivalent to the well-definedness of the linear transformation $Ax \mapsto x$ from the range of $A$ to $\mathbb{C}^n$, or to the left-invertibility of $A$.

\vspace{3mm}

(b) Consider the polar decomposition of $A$: $A=V(A^*A)^{1/2}$, where $V$ is an $m$-by-$n$ partial isometry with $\ker V=\ker A$ (cf. \cite[Problem 134]{3}). Then
\begin{align*}
& \, m(A)=\min\{\|Ax\|:x\in \mathbb{C}^n, \|x\|=1\}\\
= & \, \min\{\|V(A^*A)^{1/2}x\|:x\in \mathbb{C}^n, \|x\|=1\}\\
= & \, \min\{\|(A^*A)^{1/2}x\|:x\in \mathbb{C}^n, \|x\|=1\}\\
= & \, \mbox{minimum eigenvalue of } (A^*A)^{1/2}\\
= & \, \mbox{minimum singular value of } A.
\end{align*}\end{proof}

To prove Theorem \ref{t31} (b), we need another lemma to get around the restriction $A_1\ldots A_{k-1}\neq 0$.

\vspace{3mm}

\begin{lemma}\label{l33}
If $A_j$ is an $n_j$-by-$n_{j+1}$ matrix, $1\le j\le k-1$, such that $A_1\ldots A_{k-1}= 0$, then for any $\varepsilon>0$, there are $n_j$-by-$n_{j+1}$ matrices $B_j$ such that $\|B_j-A_j\|<\varepsilon$ for all $j$ and $B_1\ldots B_{k-1}\neq 0$.
\end{lemma}

\begin{proof}
This is proven by induction on $k$. The case of $k=2$ is trivial. We now assume that $k=3$ and $A_1A_2=0$. Consider the following four cases separately:

\vspace{3mm}

(i) $A_1=0$ and $A_2=0$. Let $B_1$ (resp., $B_2$) be the $n_1$-by-$n_2$ (resp., $n_2$-by-$n_3$) matrix with its $(1, 1)$-entry equal to $\varepsilon/2$ and all other entries 0. Then $B_1B_2$ has the $(1, 1)$-entry $\varepsilon^2/4$, and hence is nonzero.

\vspace{3mm}

(ii) $A_1\neq 0$ and $A_2=0$. Assume that $a_{ij}$, the $(i, j)$-entry of $A_1$, is nonzero. Let $B_1=A_1$ and let $B_2$ be the $n_2$-by-$n_3$ matrix with its $(j, 1)$-entry equal to $\varepsilon/2$ and all others 0. Then the $(i, 1)$-entry of $B_1B_2$ is $a_{ij}\varepsilon/2$, which is nonzero. Hence $B_1B_2\neq 0$.

\vspace{3mm}

(iii) $A_1= 0$ and $A_2\neq 0$. By symmetry, this case can be dealt with as in (ii).

\vspace{3mm}

(iv) $A_1, A_2\neq 0$. Assume that $x_i^T$, the $i$th row of $A_1$, and $y_j$, the $j$th column of $A_2$, are nonzero. Since $x_i^Ty_j=0$, we may perturb $y_j$ slightly to a column vector $z_j$ such that $x_i^Tz_j\neq 0$. Let $B_1=A_1$ and $B_2$ be obtained from $A_2$ by replacing its $y_j$ by $z_j$. Then $B_1B_2\neq 0$.

\vspace{3mm}

Note that in (ii) and (iv) above, we have actually shown that if $A_1A_2=0$ and $A_1\neq 0$, then for any $\varepsilon>0$ there is a matrix $B_2$ such that $\|B_2-A_2\|<\varepsilon$ and $A_1B_2\neq 0$. This will be used in the induction process below.

\vspace{3mm}

Assume that our assertion is true for $k-2$ and that $A_1\ldots A_{k-1}=0$. If $A_1\ldots A_{k-2}=0$, then the induction hypothesis implies, for each $\varepsilon>0$, the existence of matrices $B_1, \ldots, B_{k-2}$ such that $\|B_j-A_j\|<\varepsilon$ for $1\le j\le k-2$ and $B_1\ldots B_{k-2}\neq 0$. If $(B_1\ldots B_{k-2})A_{k-1}\neq 0$, then simply let $B_{k-1}=A_{k-1}$; otherwise, from (ii) and (iv) above, there is a matrix $B_{k-1}$ such that $\|B_{k-1}-A_{k-1}\|<\varepsilon$ and $(B_1\ldots B_{k-2})B_{k-1}\neq 0$. On the other hand, if $A_1\ldots A_{k-2}\neq 0$, then, since $(A_1\ldots A_{k-2})A_{k-1}= 0$, (ii) and (iv) above yields a matrix $B_{k-1}$ such that $\|B_{k-1}-A_{k-1}\|<\varepsilon$ and $(A_1\ldots A_{k-2})B_{k-1}\neq 0$. Letting $B_j=A_j$ for $1\le j\le k-2$ proves our assertion.
\end{proof}

We are now ready to prove Theorem \ref{t31}.

\begin{proof}[Proof of Theorem \ref{t31}]
(a) First assume that $A_1\ldots A_{k-1}\neq 0$. Since $A''$ is an (entrywise) nonnegative matrix, there is a unit vector $y=[y_1 \ \ldots \ y_k]^T$ in $\mathbb{C}^k$ with $y_j\ge 0$ for all $j$ such that $\langle A''y,y\rangle=w(A'')$ (cf. \cite[Proposition 3.3]{5}). Let $u$ be a unit vector in $\mathbb{C}^{n_k}$ such that $A_1\ldots A_{k-1}u\neq 0$, let $x_j=A_jA_{j+1}\ldots A_{k-1}u/\|A_jA_{j+1}\ldots A_{k-1}u\|$ for $1\le j\le k-1$ and $x_k=u$, and let $v=[y_1x_1 \ \ldots \ y_kx_k]^T$. Then $v$ is a unit vector in $\mathbb{C}^n$ since
\[\|v\|=\big(|y_1|^2\|x_1\|^2+\cdots+|y_k|^2\|x_k\|^2\big)^{1/2}=\big(|y_1|^2+\cdots+|y_k|^2\big)^{1/2}=1.\]
Moreover,
\[\langle Av,v\rangle=\sum_{j=1}^{k-1}\langle A_j(y_{j+1}x_{j+1}), y_jx_j\rangle=\sum_{j=1}^{k-1}y_{j+1}y_j\langle A_jx_{j+1},x_j\rangle.\]
Note that
\begin{align*}
& \, \langle A_jx_{j+1},x_j\rangle = \Big\langle\frac{A_jA_{j+1}\ldots A_{k-1}u}{\|A_{j+1}\ldots A_{k-1}u\|}, \frac{A_j\ldots A_{k-1}u}{\|A_j\ldots A_{k-1}u\|}\Big\rangle\\
= & \, \frac{\|A_j\ldots A_{k-1}u\|}{\|A_{j+1}\ldots A_{k-1}u\|} \ge m(A_j).
\end{align*}
Hence
\begin{equation}
\langle Av,v\rangle\ge \sum_{j=1}^{k-1}y_{j+1}y_jm(A_j)=\langle A''y, y\rangle=w(A''). \label{eq5}
\end{equation}
It follows that $w(A)\ge w(A'')$ as asserted.

\vspace{3mm}

Now if $A_1\ldots A_{k-1}=0$, then, for any $\varepsilon>0$, let $B_1, \ldots, B_{k-1}$ be as in Lemma \ref{l33}, and let
\[B=\left[\begin{array}{cccc} 0 & B_1 & & \\ & 0 & \ddots & \\ & & \ddots & B_{k-1} \\ & & & 0\end{array}\right]  \mbox{ on } \mathbb{C}^n \mbox{ and }
B''=\left[\begin{array}{cccc} 0 & m(B_1) & & \\ & 0 & \ddots & \\ & & \ddots & m(B_{k-1}) \\ & & & 0\end{array}\right]  \mbox{ on } \mathbb{C}^k.\]
From the first half of the proof, we have $w(B)\ge w(B'')$. Since
\[\|A''-B''\|=\max_j|m(A_j)-m(B_j)|\le\max_j\|A_j-B_j\|<\varepsilon\]
(cf. \cite[Lemma 2.2 (1)]{10}), we infer from the continuity of $w(\cdot)$ that $w(A)\ge w(A'')$ (cf. \cite[Problem 220]{3}). This completes the proof of (a).

\vspace{3mm}

(b) Assume that $A_1\ldots A_{k-1}\neq 0$ and $w(A)=w(A'')$. Let $y=[y_1\ \ldots \ y_k]^T\in \mathbb{C}^k$, $u\in \mathbb{C}^{n_k}$, $x_j\in\mathbb{C}^{n_j}$ for $1\le j\le k$, and $v\in \mathbb{C}^n$ be as in the first half of the proof of (a). Let $\widehat{x}_j=[0 \ \ldots \ 0 \ \stb{x_j}{j{\rm th}} \ 0 \ \ldots \ 0]^T$ for $1\le j\le k$, and let $K$ be the subspace of $\mathbb{C}^n$ spanned by the $\widehat{x}_j$'s. Since $A\widehat{x}_1=0$ and
\begin{equation}
A\widehat{x}_j=\big[0 \ \ldots \ 0 \ \btb{\frac{A_{j-1}A_j\ldots A_{k-1}u}{\|A_j\ldots A_{k-1}u\|}}{(j-1)\mbox{st}} \ 0 \ \ldots \ 0]^T=\frac{\|A_{j-1}\ldots A_{k-1}u\|}{\|A_j\ldots A_{k-1}u\|}\widehat{x}_{j-1}, \ 2\le j\le k, \label{eq6}
\end{equation}
we obtain $AK\subseteq K$.

\vspace{3mm}

We next show that $A^*K\subseteq K$. Indeed, since $w(A)=w(A'')$, we have an equality in \eqref{eq5}, which yields that $\|A_{j}\ldots A_{k-1}u\|/\|A_{j+1}\ldots A_{k-1}u\|=m(A_j)$ for all $j$, $1\le j\le k-1$. This is because $A_j\neq 0$ for all $j$ and thus $A''$ is a nonnegative matrix with irreducible real part, from which we infer that $y_j>0$ for all $j$ (cf. \cite[Proposition 3.3]{5}). Since the $x_j$'s are unit vectors satisfying $\|A_jx_{j+1}\|=m(A_j)$, we have $\langle (A_j^*A_j-m(A_j)^2I_{n_{j+1}})x_{j+1}, x_{j+1}\rangle=0$, $1\le j\le k-1$. From $A_j^*A_j\ge m(A_j)^2I_{n_{j+1}}$, we infer that $A_j^*A_jx_{j+1}=m(A_j)^2x_{j+1}$ and hence $A_j^*A_jA_{j+1}\ldots A_{k-1}u=m(A_j)^2A_{j+1}\ldots A_{k-1}u$. This shows that $A_j^*x_j$ is a multiple of $x_{j+1}$ and thus $A^*\widehat{x}_j$ is a multiple of $\widehat{x}_{j+1}$ for $1\le j\le k-1$. Therefore, $A^*\widehat{x}_j$ is in $K$ for all $j$, $1\le j\le k-1$. Together with $A^*\widehat{x}_k=0$, these imply that $A^*K\subseteq K$. Hence $A$ is unitarily similar to $(A|K)\oplus(A|K^{\perp})$. Since $A\widehat{x}_1=0$ and $A\widehat{x}_j=m(A_{j-1})\widehat{x}_{j-1}$, $2\le j\le k$, from \eqref{eq6}, we have the unitary similarity of $A|K$ and $A''$. On the other hand, the unitary similarity of $A|K^{\perp}$ to a block shift follows as in the last part of the proof of Theorem \ref{t21} (b).
\end{proof}

\vspace{1mm}

\begin{corollary}\label{c34}
Let $A$ be an $n$-by-$n$ block shift as in \eqref{eq1}, and let $m=\min_jm(A_j)$. Then

{\rm (a)} $w(A)\ge m\cdot\cos(\pi/(k+1))$, and

{\rm (b)} $w(A)=m\cdot\cos(\pi/(k+1))$ if and only if $A$ is unitarily similar to $(m\cdot J_k)\oplus C$, where $C$ is a block shift with $w(C)\le m\cdot\cos(\pi/(k+1))$.
\end{corollary}

\vspace{3mm}

This can be proven as Corollary \ref{c22} by using Theorem \ref{t31} and \cite[Lemma 5]{8}.

\vspace{3mm}

Analogous to the situation in Section \ref{s2}, Theorem \ref{t31} (b) remains true for $n\le 3$ without the assumption $A_1\ldots A_{k-1}\neq 0$. This is no longer the case for $n\ge 4$. A counterexample for $n=4$ is given below.

\vspace{3mm}

\begin{example}
Let
\[A=\left[\begin{array}{cccc} 0 & 1 & 1 & \\ & 0 & 0 & 1\\ & 0 & 0 & -1 \\ & & & 0\end{array}\right]\]
with $A_1=[1 \ 1]$ and $A_2=\left[\begin{array}{c} 1\\ -1\end{array}\right]$. In this case, $A_1A_2=[0]$ and
\[A''=\left[\begin{array}{ccc} 0 & 0 &  \\ & 0 & \sqrt{2} \\  & & 0\end{array}\right].\]
Since $A^2=0$, we have $w(A)=\|A\|/2=\sqrt{2}/2$ (cf. \cite[Theorem 2.1]{9}). On the other hand, we also have $w(A'')=\sqrt{2}/2$. But $A''$ is not a direct summand of $A$. This is because if it is, then $A$ will be unitarily similar to $A''\oplus[0]$, which is impossible since $\ker A\cap\ker A^*=\{0\}$.
\end{example}

\vspace{3mm}

A larger parameter than the minimum modulus of an $m$-by-$n$ matrix $A$ is its \emph{reduced minimum modulus} $\gamma(A)$ defined by
\[\gamma(A)=\left\{\begin{array}{ll} \min\{\|Ax\|:x\in \mathbb{C}^n, x\perp\ker A, \|x\|=1\} \ \ & \mbox{ if } \ A\neq 0,\\
0 & \mbox{ if } \ A= 0.\end{array}\right.\]
A general reference for $\gamma(A)$ (when $A$ is an operator on a possibly infinite-dimensional Hilbert space) is \cite{1}. For an $n$-by-$n$ block shift $A$ of the form \eqref{eq1}, consider the $k$-by-$k$ matrix
\[A'''=\left[\begin{array}{cccc} 0 & \gamma(A_1) &  & \\ & 0 & \ddots & \\ &  & \ddots & \gamma(A_{k-1}) \\ & & & 0\end{array}\right].\]
We may expect to have $w(A''')$ as a lower bound for $w(A)$ under some extra conditions on $A$. The next theorem shows that this is indeed the case for small values of $k$.

\vspace{3mm}

\begin{theorem}\label{t36}
{\rm (i)} Let $A=\left[\begin{array}{cc} 0 & A_1  \\ 0 & 0\end{array}\right]$ and $A'''=\left[\begin{array}{cc} 0 & \gamma(A_1)  \\ 0 & 0\end{array}\right]$. Then

{\rm (a)} $w(A)\ge w(A''')$, and

{\rm (b)} $w(A)=w(A''')$ if and only if $A$ is unitarily similar to $A'''\oplus\cdots\oplus A'''\oplus 0$.

\vspace{3mm}

{\rm (ii)} Let
\[A=\left[\begin{array}{ccc} 0 & A_1  \\ & 0 & A_2\\ & & 0\end{array}\right] \mbox{on } \mathbb{C}^n=\mathbb{C}^{n_1}\oplus \mathbb{C}^{n_2}\oplus \mathbb{C}^{n_3} \mbox{ and } A'''=\left[\begin{array}{ccc} 0 & \gamma(A_1) &   \\  & 0 & \gamma(A_{2}) \\  & & 0\end{array}\right] \mbox{on } \mathbb{C}^3.\]
Assume that ${\rm rank} \, A_1+ {\rm rank} \, A_2>n_2$. Then

{\rm (a)} $w(A)\ge w(A''')$, and

{\rm (b)} $w(A)=w(A''')$ if and only if $A$ is unitarily similar to $A'''\oplus C$, where $C$ is a block shift with $w(C)\le w(A''')$.
\end{theorem}

\vspace{3mm}

The proof is quite similar to the one for Theorem \ref{t31}, which we omit. For larger values of $k$, the extra conditions on the $A_j$'s are two cumbersome to be of any practical use.

\vspace{3mm}

\section*{\bf Acknowledgments}
The two authors acknowledge supports from the Ministry of Science and Technology of the Republic of China under projects MOST 103-2115-M-008-006 and NSC 102-2115-M-009-007, respectively. The second author was also supported by the MOE-ATU.

\vspace{3mm}

\end{document}